\documentclass[smallextended]{svjour3}       
\smartqed  
\usepackage{graphicx}
\usepackage[utf8]{inputenc}
\usepackage[export]{adjustbox}
\usepackage{wrapfig}

\usepackage{amsmath}
\usepackage{amsfonts}
\usepackage{float}
\usepackage{color}
\usepackage{xcolor}
\newcommand{\bv}{\bar{V}}
\newcommand{\hv}{\hat{V}}
\newcommand{\hi}{\hat{I}}
\newcommand{\hp}{\hat{v}}

\newcommand{\dd}{\,\mathrm{d}}

\begin{document}

\title{Dynamics of SIR model with vaccination and heterogeneous	behavioral response of individuals modeled by the Preisach operator
	\thanks{The work of Jana Kopfov\'a and Petra  Petra N\'ab\v elkov\'a  
		was supported by the institutional support for the development of research organizations I\v C 47813059.}
}
%



\author{
        Jana Kopfov\'a          \and
        Petra N\'ab\v elkov\'a \and
	Dmitrii Rachinskii      \and
        Samiha C. Rouf
}


\institute{  J. Kopfov\'a \at
              Mathematical Institute of the Silesian University, NaRybn\' i\v cku 1, 746 01 Opava, Czech Republic\\
              \email{Jana.Kopfova@math.slu.cz}       
           \and
           P. N\'ab\v elkov\'a \at
              Mathematical Institute of the Silesian University, NaRybn\' i\v cku 1, 746 01 Opava, Czech Republic\\
              \email{petra.nabelkova@math.slu.cz}
	\and
           D. Rachinskii \at
              University of Texas at Dallas, 800 W. Campbell, Richardson, Texas, 75080, United States\\
              \email{Dmitry.Rachinskiy@utdallas.edu}
	\and
           S. Rouf \at
              University of Texas at Dallas, 800 W. Campbell, Richardson, Texas, 75080, United States\\
              \email{samiha.rouf@utdallas.edu}
}

\date{Received: date / Accepted: date}

\maketitle

\begin{abstract}
\noindent   
We study global dynamics of an SIR  model with vaccination, where we assume that individuals respond differently to dynamics of the epidemic. Their heterogeneous response is modeled by the Preisach hysteresis operator. We present a condition for the global stability of the infection-free equilibrium state. If this condition does not hold true, the model has a connected set of endemic equilibrium states characterized by different proportion of infected and immune individuals. In this case, we show that every trajectory converges either to an endemic equilibrium or to a periodic orbit. Under additional natural assumptions, the periodic attractor is excluded, and we guarantee the convergence of each trajectory to an endemic equilibrium state. The global stability analysis  uses a family of Lyapunov functions corresponding to the family of branches of the hysteresis operator.
\end{abstract}

\section{Introduction} 
In the classical compartmental models of epidemiology, the key parameters such as the transmission and vaccination rates are assumed to remain the same during an epidemic. As an example, consider a SIR model with vaccination of the form
\begin{equation}
\begin{aligned}
\label{SIR}
\dot S & = - \frac{\beta S I}{N} -v S - \mu S +\mu N,\\
\dot I & =\frac{\beta S I}{N} - \gamma I-\mu I, \\
\dot R & = \gamma I +v S- \mu R,\\
\end{aligned}
\end{equation}
where $S$, $I$ and $R$ are densities of the populations of susceptible, infected and recovered individuals, respectively;
$N=S+I+R$ is the total population density;
$\beta$ is the transmission rate, which can be expressed as
the product of the average number of daily contacts a susceptible individual has with other individuals and the probability of transmission during each contact; $v$ is the vaccination rate;  
$\gamma$ is the recovery rate; and,
$\mu$ is the birth/mortality rate
(or  immigration/emigration rate; {for childhood infections, individuals leave the group at a certain age}). The recovered individuals are assumed to be immune to the disease. 
In this particular model setup, the equality of the birth and mortality rates ensures the conservation of the total population density, i.e.\ $N$ is constant.

The key parameter is the so-called {basic reproduction number} $R_0$. For model \eqref{SIR} it equals $R_0=\beta \mu/((\gamma+\mu)(\mu+v))$. If 
$R_0 < 1$, the infection dies out in the long run; if
$R_0 > 1$,
the infection  spreads in the population (Korobeinikov et al. 2002; Korobeinikov et al. 2004; Ullah et al. 2013 ) 

Due to their simplicity, the standard SIR model and its variants, including \eqref{SIR}, assume that the hosts are unable to respond in any way to the advent of an epidemic and disregard the ability of the community to adapt its behavior to the danger. Humans, however, are able to intelligently respond to a threat as they receive and perceive information regarding the epidemic from the ``outside'' world, the government and health authorities, and can adjust their behavior to avoid or to reduce the risk of being infected. Typical aspects of this adaptability may include simple precautionary measures, such as refraining from potentially dangerous contacts, increasing hygiene, using hand sanitizer and disinfectants, wearing face mask in public places, adjusting a general life style, taking an extra portion of vitamin C in a case of a common cold, or using vaccination in a case of influenza. 
At the threat of epidemics, the government and health authorities can intervene by promoting immunizations, if available, raising awareness in the population about the current severity of the epidemic, providing access to effective and affordable medicines and tests, working with school authorities, using media and/or administrative pressure, etc.
During the covid-19 epidemics we have seen such interventions imposed by the authorities on an unprecedented scale including massive quarantine and social distancing measures, business restrictions, gathering limitations, stay-at-home and  shelter-in-place orders and closing the state borders.

In order to account for the adaptability of the population, the assumptions of the standard SIR models, which postulate constant transmission and vaccination rates during the whole time period in question, was revisited in several different ways.
The incidence rate of the modified form ${\beta S I}/({1 + a I})$ or ${\beta S I}/({1 + a S})$ was used to account for the saturation effect with saturation rate $a$. The first form is based on the assumption that  an increase in the number of infective individuals leads to a reduction of the incidence rate; the second form is associated with protective measures
taken by susceptible individuals against the infection. 
The two effects were also combined  into the incidence rate of the form   ${\beta S I}/({1 + a S + bI})$ (Dubey et al. 2015, Kaddar 2010) 
for an overview of several models. Their extensions of these models include non-pharmaceutical intervention factors such as quarantine and isolation of patients (Hou et al. 2020; Davies et al. 2020; Volpert et al 2020; Wearing et al. 2005). 
Time-dependent transmission rates were used to account for seasonal effects and varying weather conditions (Grassly et al. 2006; Liu et al. 2012).

Immunization is a proven and probably the most effective tool for controlling and eliminating infectious diseases.
It was hypothesized that  
a vaccination effort 
can be more efficient when it
is pulsed in time rather than uniform. 
The effect of pulsed vaccination policy has been studied quite extensively (Agur et al. 1993; Lu et al. 2002)where a detailed comparison of models with constant and pulsed vaccination rate is provided. 
Piecewise smooth epidemiological models of switched vaccination, implemented once the number of people exposed to a virus reaches a critical level, were studied in Wang et al. (2014). 
In yet another class of epidemiological models with adaptive switching behavior,  a stochastic switching model was combined with an economic optimal stopping problem to determine the optimal timings for public health interventions (Sims et al. 2016). Models studied in Chladn\'a et al. (2020) assume that intervention measures are implemented when the number of infected individuals exceeds a critical level, and the intervention stops when this number drops below a different (lower) threshold. Implications of such two-threshold intervention strategies for dynamics of an SIR model were considered both in the case of switched vaccination rate (see Section \ref{2.1}) and switched transmission rate.

Multiple factors can influence the willingness of an individual to receive immunization depending on the perceived risk of contracting the disease, risk of possible complications, personal beliefs, etc. These risks vary with age, health conditions, lifestyle and profession. Further, interventions of the health authorities and administrative measures at the level of a county or state can vary in scale depending on the availability of resources, the local economic situation and other factors. All these variations lead to the heterogeneity of the response of the population to the advent of an epidemic.

In this paper, we propose a variant of the SIR model with a heterogeneous response and analyze how the heterogeneity affects the dynamics. We focus on the scenario when interventions of the health authorities affect the vaccination rate (assuming vaccines are available) while the transmission rate remains constant; the case of variable transmission rate will be considered elsewhere.
As a starting point, we adopt the approach of Chladn\'a et al. (2020)  to modeling the homogeneous switched response of a subpopulation to the varying number of infected individuals by a two-threshold two-state relay operator (as described in Section \ref{2.1}). To reflect the heterogeneity, multiple subpopulations are considered, each characterized by a different pair of switching thresholds. In order to keep the model relatively simple, we apply averaging under further simplifying assumptions. The main simplification is that perfect mixing of the population is assumed. This leads to a differential model with just two variables, $S$ and $I$, but with a complex operator relationship between the vaccination rate $v$ and the density of the infected individuals $I$.  As such, this operator relationship, known as the Preisach operator, accounts for the heterogeneity of the response.
Heterogeneity of intervention policies can be modeled in a similar fashion
(see Section \ref{2.2}).

We show that the system with the Preisach operator is amenable to analysis 
when interpreted as a switched system (Bernardo et al. 2008) associated with a one-parameter family of nonlinear planar vector fields
$\Phi_u=\Phi_u(I,S)$ (where $u\in \mathbb{R}$ is a parameter).
Between the switching moments, a trajectory of the system is an integral curve of a particular vector field. The Preisach operator imposes non-trivial rules for switching from one vector field  to another. 
Some intuition can be drawn from dynamics of systems with dry friction 
such as models of presliding friction behavior (Al-Bender et al. 2004; Ruderman 2011), and population models with theresholds (Meza et al. 2009).

Using the switched systems approach, we show that if $R_0\le 1$, then the infection-free equilibrium is globally stable. In the case of $R_0>1$, the bi-stable nature of an individual response leads to multi-stability in the aggregated model. Namely, the system has a connected set of endemic equilibrium states characterized by different proportion of infected and immune individuals\footnote{Using an analogy with mechanical systems that exhibit dry friction this is not surprising: for example, an object can achieve an equilibrium on a curved surface at any point where the slope does not exceed the dry friction coefficient because friction balances the gravity. On the other hand, this is not unlike the classical SIR model \eqref{SIR} with zero mortality and vaccination rates rates, $\mu=v=0$, where the infection-free equilibrium states form the segment $S+R=N$, $0\le S\le N$.}. In this case, we show that every trajectory converges either to one of the endemic equilibrium states or to a periodic orbit corresponding to the recurrence of the disease. Under additional natural assumptions,
we prove the global stability of the set of endemic equilibrium states
by adapting the method of Lyapunov functions.  Each vector field $\Phi_u$ has a global Lyapunov function $V_u=V_u(I,S)$. We 
establish the global stability of the switched system by
 controlling the increment $V_u(I_1,S_1)-V_u(I_2,S_2)$ of the Lyapunov function along a trajectory between the switching points, and the difference $V_{u_1}(I,S)-V_{u_2}(I,S)$ of the Lyapunov functions at those points. 
  Numerical analysis of SIR models with the Preisach operator was previously performed in (Pimenov et al. 2010, Pimenov et al. 2012) 

The paper is organized as follows. In the next section we present the model
and remind the definition of the Preisach operator.
In Section 3 
some preliminary properties of the Preisach operator and the model are discussed, including
hysteresis loops, equilibrium states and the global stability of the infection-free equilibrium in the case $R_0\le 1$. Sections 4 and 5 present the main results on dynamics in the case $R_0>1$ and their proofs.


\section{Model}
	
We consider the following SIR model 		
\begin{equation}\label{2a}
\begin{array}{l}
\dot I=\beta IS -(\gamma+\mu)I,\\
\dot S= -\beta IS -v(t)S-\mu S+\mu,\\
\dot R=\gamma I +v(t) S -\mu R
\end{array}
\end{equation}
with an additional feedback loop, which relates the variable vaccination rate $v=v(t)$ to the concurrent and past values of the density $I=I(t)$ of the infected population. In the main part of the paper, it is assumed that the function $I: \mathbb{R}_+\to \mathbb{R}$ is mapped to the function $v: \mathbb{R}_+\to \mathbb{R}$ by the so-called continuous Preisach operator, which is described in the following sections.
In order to motivate and explain the nature of the assumed operator
 relationship between $I$ and $v$, we first briefly discuss the non-ideal relay operator 
 in the same context. {Regardless} of the specific form of the feedback,
 	one can see that the sum $I+S+R$ is conserved by system \eqref{2a}, and the last equation is redundant. 
 	Without loss of generality, we can interpret $S, I, R$ as relative densities
 	assuming that $I+S+R=1$
 	at all times.
 We denote $\delta = \mu + \gamma$ and rewrite the system as
 \begin{equation}\label{2}
 \begin{array}{l}
 \dot I=\beta IS - \delta I,\\
 \dot S= -\beta IS -v(t)S-\mu S+\mu.
 \end{array}
 \end{equation}

Note that the domain
 \begin{equation}\label{D}
 \mathfrak{D}=\{ (I,S): I>0, S>0, S+I \le 1\}
 \end{equation} 
 is flow-invariant for this system. 
 Indeed, in this region, $\dot I\le (\beta-\delta)I$ and 
 $$
 \begin{array}{c}\dot I=0 \quad \text{for} \quad I=0,\\
 \dot S=\mu>0 \quad \text{for} \quad S=0,\\
 \dot I + \dot S =  - \delta I - v(t) S + \mu (1-S) \leq -\delta I + \mu - \mu S < 
 0 \quad \text{for} \quad {I+S=1}
 \end{array}
 $$
 (where we use $\delta>\mu$),
 which implies  the statement. 
 We will consider trajectories from the domain $\mathfrak{D}$ only.
 
  \subsection{Switched model with one non-ideal relay {operator}}\label{2.1}
 In Chladn\'a et al. 2020, the relationship between the density of the infected population, $I=I(t)$,
 and the vaccination rate 
 was modeled 
 by
 the simplest hysteretic 
 operator called the {\em non-ideal relay}, which is
 also known as {a rectangular hysteresis loop} or a lazy switch (Visintin 1994).
 The relay operator is characterized by two scalar parameters $\alpha_1$ and $\alpha_2$, the threshold values, with $\alpha_1<\alpha_2$. We will use the notation $\alpha=(\alpha_1,\alpha_2)$. The input of the relay is an arbitrary continuous function of time, $I: \mathbb{R}_+\to\mathbb{R}$. The {\em state} $\nu_{\alpha}(\cdot)$ equals either $0$ or $1$ at any moment $t\in \mathbb{R}_+$. If the input value at some instant is below the lower threshold {value} $\alpha_1$, then the state at this instant is $0$ and it remains equal to $0$ as long as the input is below the upper threshold value $\alpha_2$. When the input reaches the value $\alpha_2$, the state switches instantaneously to the value $1$. Then, the state remains equal to $1$ as long as the input stays above the lower threshold value $\alpha_1$. When the input reaches the value $\alpha_1$, the state switches back to $0$. This dynamics is captured by the input-state diagram shown in Figure~\ref{fig1}. 
 In particular, the input-state pair
 $(I(t),\nu_{\alpha}(t))$ belongs to the union of the two horizontal rays shown in bold in Figure
 \ref{fig1} at all times. 
 
 The above description results in the following definition.
 Given any continuous input $I: \mathbb{R}_+\to \mathbb{R}$ and an initial value of the state, $\nu_{\alpha}(0)=\nu_{\alpha}^0$,
 satisfying the constraints
 \begin{equation}\label{v1}
 \nu_{\alpha}^0\in\{0,1\} \quad \text{if} \quad \alpha_1<I(0)< \alpha_2;
 \end{equation}
 \begin{equation}\label{v2} 
 \nu_{\alpha}^0=0 \quad \text{if} \quad I(0)\le \alpha_1;\qquad \nu_{\alpha}^0=1 \quad \text{if} \quad I(0)\ge\alpha_2,
 \end{equation}
 the state of the relay at the future moments $t>0$ is defined by
 \begin{equation}\label{relay'}
 \nu_{\alpha}(t)=
 \left\{
 \begin{array}{cl}
 0 & \text{if there is $t_1\in [0,t]$ such that $I(t_1)\le \alpha_1$}\\
 & \text{and $I(\tau)<\alpha_2$ for all $\tau\in (t_1, t]$;}\\
 1 & \text{if there is $t_1\in [0,t]$ such that $I(t_1)\ge \alpha_2$}\\
 & \text{and $I(\tau)>\alpha_1$ for all $\tau\in (t_1, t]$};\\
 \nu_{\alpha}(0) & \text{if $\alpha_1<I(\tau)<\alpha_2$ for all $\tau\in[0,t]$.}
 \end{array}\right.
 \end{equation}
 This time series  of the state, which depends both on the input $I(t)$ $(t\geq 0)$ and the initial state  $\nu_{\alpha}(0)=\nu_{\alpha}^0$  of the relay, will be denoted by
 \begin{equation}\label{re}
 \nu_{\alpha}(t)=({\mathcal R}_{\alpha}[\nu_{\alpha}^0]I)(t),\qquad t\geq 0.
 \end{equation}
 By definition, the state satisfies the constraints
 \begin{equation}\label{compatibility}
 \nu_{\alpha}(t)=1 \quad \text{whenever} \quad I(t)\geq
 \alpha_2; \quad \nu_{\alpha}(t)=0 \quad \text{whenever} \quad I(t)\leq \alpha_1
 \end{equation}
 at all times. Further,
 the function \eqref{relay'} has at most a finite number of jumps between the values $0$ and $1$
 on any finite time interval $t_{0}\leq t\leq t_{1}$. 
 If the input oscillates between two values $I_1,I_2$, such that $I_1<\alpha_1<\alpha_2<I_2$, then the point $(I(t),\nu_{\alpha}(t))$ moves counterclockwise along the {rectangular hysteresis loop} shown in Figure~\ref{fig1}. 
 
 \begin{figure}[h]
 	\begin{center}
 		\includegraphics*[width=0.7\columnwidth]{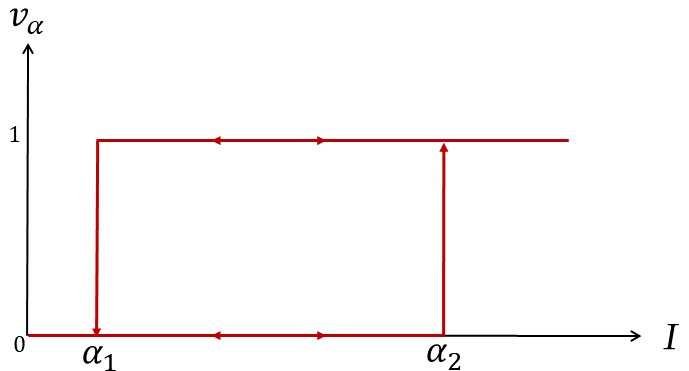}
 	\end{center}
 	\caption{The non-ideal relay operator defined by \eqref{relay'} 
 		maps the pair $(I(t),\nu_{\alpha}(0))$, where $I(t)$ ($t\ge 0$) is the input and $\nu_{\alpha}(0)$ is the initial state,  to the time series of the state  $\nu_{\alpha}(t)$ for $t>0$. The input-state pair $(I,\nu_{\alpha})$ belongs to the union of the two bold (open) horizontal rays at all times. Initially, it belongs to the upper ray if $\nu_{\alpha}(0)=1$ and to the lower ray if $\nu_{\alpha}(0)=0$. 
 		The point $(I,\nu_{\alpha})$ moves horizontally left whenever $\dot I<0$ and right whenever $\dot I>0$.
 		Further, when $(I,\nu_{\alpha})$ reaches the end of either ray, it  transits vertically to the other ray. This transition is instantaneous.
 	}
 	\label{fig1}
 \end{figure}

In Chladn\'a et al. (2020), it was assumed that interventions of the health authority change the vaccination rate according to the following rules.
The vaccination rate is switched from a lower rate $v_{nat}$ to a higher rate ${v_{int}:=v_{nat}+ q_0}$, ${q_0 > 0}$, when the density of the infected population reaches a threshold value $\alpha_2$. The intervention stops when the number of infected individuals drops below a lower threshold value $\alpha_1$, at which point the vaccination rate returns to its lower value $v_{nat}$.
Using the definition \eqref{relay'} of the non-ideal relay operator \eqref{re}, this leads to the formula
\begin{equation}\label{rev}
v(t)= v_{nat} 
+ q_0\cdot
	({\mathcal R}_{\alpha}[\nu_{\alpha}^0]I)(t).
\end{equation}
Coupling of this operator equation with dynamic equations \eqref{2} results in a switched system.

As  shown in Chladn\'a et al. (2020), {switched} system \eqref{2}, \eqref{rev} exhibits different dynamic scenarios depending on its parameters.
In particular, it can have a globally stable endemic equilibrium. Alternatively, a locally stable endemic equilibrium coexists with a stable periodic orbit. Along this orbit, the vaccination rate \eqref{rev} switches between the values $v_{nat}$ and $v_{nat}+ q_0$ twice per period. 

\subsection{Model with {heterogeneous} vaccination rate}\label{2.2}
%

Now, we consider a model, in which several vaccination laws of the form \eqref{rev}, with different thresholds $\alpha$, are combined either because the health authority employs multiple intervention strategies or because different individuals respond differently to the advent of an epidemic.

Assume that the health authority 
has multiple intervention policies (numbered $n=1,\ldots,N$) 
in place, each increasing the vaccination rate by a certain amount $ q_n$ while the intervention is implemented,
in order to provide a response, which is adequate to the severity of the epidemic. 
Further, assume that each  intervention policy 
is guided by the two-threshold start/stop rule, such as in \eqref{rev}, associated with a particular pair of thresholds $\alpha^n=(\alpha_1^n,\alpha_2^n)$. Under these assumptions,  the vaccination rate in system \eqref{2} is given by
\begin{equation}\label{rev'}
v(t)= v_{nat} 
+ \sum_{n=0}^{N-1} q_n\cdot
({\mathcal R}_{\alpha^n}[\nu_{\alpha^n}^0]I)(t).
\end{equation}
This formula defines a mapping from the space of continuous
inputs $I:\mathbb{R}_+\to \mathbb{R}$ to the space of piecewise constant outputs $v:\mathbb{R}_+\to \mathbb{R}$, which is known as the discrete Preisach operator ( Krasnosel'skii et al. 1983), see Figure \ref{fig2}. In particular, the vaccination rate \eqref{rev'} is set to change at multiple thresholds $\alpha_1^n, \alpha_2^n$.

\begin{figure}[ht]
	\begin{center}
		\includegraphics*[width=0.9\columnwidth]{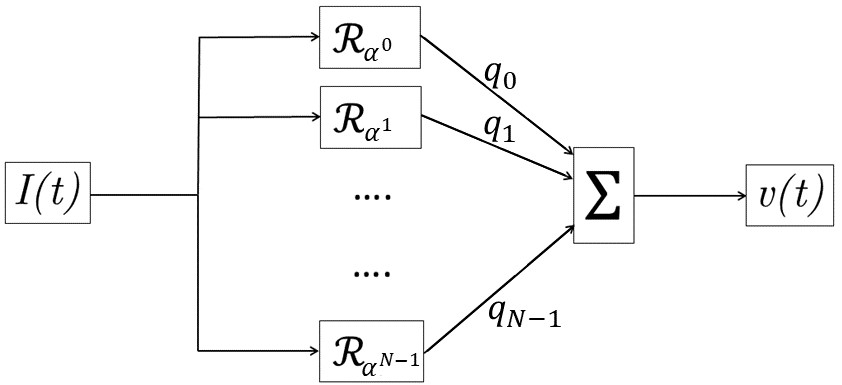}
	\end{center}
	\caption{ Preisach model as the parallel connection of non-ideal
		relays with weights. Relays $R_{\alpha}$ with different pairs of thresholds respond to a common input $I(t)$. These
		relays function independently of each other and contribute to the
		output $v(t)$ of the model, which is defined as the weighted
		sum (integral) of the outputs of the individual relays
		$R_{\alpha}$.
	}
	\label{fig2}
\end{figure}

On the other hand, individuals can respond differently to 
dynamics of the epidemic 
and interventions of the health authority.
In particular, the willingness to receive vaccination can vary significantly from individual to individual for the same level of threat of contracting the disease.
In order to account for the heterogeneity of the individual response, let us divide the susceptible population into non-intersecting subpopulations parametrized by points $\alpha$ of a subset $\Pi\subset \{\alpha=(\alpha_1,\alpha_2): \alpha_1<\alpha_2\}$ of the $\alpha$-plane. Assuming that the vaccination rate for a subpopulation labeled $\alpha$  is given by \eqref{rev} with $q=q(\alpha)$, the total vaccination rate equals
\begin{equation}\label{rev''}
v(t)= v_{nat} 
+ \iint_{\Pi} q(\alpha)\,
({\mathcal R}_{\alpha}[\nu_{\alpha}^0]I)(t)\, d F(\alpha),
\end{equation}
where the probability measure $F$ describes the distribution of the susceptible {population} over the index set $\Pi$ (the set of threshold pairs). As a simplification, let us assume that this measure is independent of time (in particular, the distribution does not change with variations of $I$). Then, 
the mapping of the space of continuous
inputs $I:\mathbb{R}_+\to \mathbb{R}$ to the space of outputs $v:\mathbb{R}_+\to \mathbb{R}$
defined by \eqref{rev''} is known as the general Preisach operator, which includes the discrete Preisach operator \eqref{rev'} and a continuous Preisach model (corresponding to an absolutely continuous measure $F$) as particular cases.
In either case, 
the Preisach operator is referred to as a superposition (or parallel connection) of weighted non-ideal relays.


\subsection{Continuous Preisach model}
Let us consider a rigorous definition of the Preisach operator \eqref{rev''}
with an absolutely continuous measure $F$ (Krasnosel'skii et al., 1989). 
It involves a collection of non-ideal relays ${\mathcal R}_{\alpha}$, which respond to the same  continuous input $I=I(t)$ independently. The relays contributing to the system have different pairs of thresholds $\alpha=(\alpha_1,\alpha_2)\in \Pi$,
where we
assume that $\Pi$ is measurable and bounded;
the $\alpha$-plane is called the Preisach plane.
The output 
of the continuous Preisach model is the scalar-valued function 
\begin{equation}\label{pre}
v(t)=v_{nat}+\iint_{\Pi} q(\alpha)\, \bigl({\mathcal R}_{\alpha}[\nu_{\alpha}^0]I\bigr)(t) \,d\alpha_1 d\alpha_2,\qquad t\ge 0,
\end{equation}
where 
$q:\Pi\to\mathbb{R}$  is a positive bounded measurable function (measure density) representing the weights of the relays; and, $\nu_{\alpha}^0$
is the initial state of the relay ${\mathcal R}_{\alpha}$ for any given $\alpha\in\Pi$.  
The function $\nu^0=\nu_{\alpha}^0:\Pi\to \{0,1\}$ of the variable $\alpha=(\alpha_1,\alpha_2)$
is referred to as the {\em initial state} of the Preisach operator. It
is assumed to be measurable and satisfy the constraints \eqref{v1}, \eqref{v2}, in which case the initial state-input pair is called {\em compatible}.
These requirements ensure that the integral in \eqref{pre} is well-defined for each $t\ge 0$ and, furthermore, the output $v(\cdot)$ of the Preisach model is continuous.
The input-output operator of the Preisach model defined by \eqref{pre} will be denoted by
\begin{equation}\label{P}
v(t)=({\mathcal P}[\nu^0]I)(t),\qquad t\geq 0,
\end{equation}
where both the input $I:\mathbb{R}_+\to \mathbb{R}$ and the initial state $\nu^0=\nu^0_{\alpha}$ (which is compatible with the input)
are the arguments; the value of this operator is the output
$v:\mathbb{R}_+\to \mathbb{R}$. 


In what follows we consider system \eqref{2} with the vaccination rate defined by equation \eqref{pre}.

\section{Preliminaries}

We begin by discussing 
some of 
the properties of the Preisach operator and  system \eqref{2}, \eqref{pre}.

\subsection{Global Lipschitz continuity.} 
The Preisach operator \eqref{P} is globally Lipschitz continuous (Krasnosel'skii et al. 1983). More precisely, the relations 
\[
v_k(t)=({\mathcal P}[\nu^0_k]I_k)(t),\qquad t\geq 0, {\qquad k=1,2,}
\]
imply
\begin{equation}\label{LipP*}
\|v_1-v_2\|_{C([0,\tau];\mathbb{R})} \le  K \Big( \|\nu^0_1-\nu^0_2\|_{L_1(\Pi;\mathbb{R})} + \|I_1-I_2\|_{C([0,\tau];\mathbb{R})}\Big)
\end{equation}
for any $\tau\geq 0$ with
\begin{equation}\label{K}
K := \max_{0\leq \alpha_1\leq 1} \int_{\alpha_1}^1 q(\alpha_1,\alpha_2) \dd \alpha_2.
\end{equation}

	Let us denote by $\mathfrak U$ the set of all triplets $(I_0,S_0,\nu^0)$, where $(I_0,S_0)\in \mathfrak D$ (see \eqref{D})
	and the initial state $\nu^0$ of the Preisach operator is compatible with $I_0$. 
	The global Lipschitz estimate \eqref{LipP*} ensures (for example, using the Picard-Lindel\"of type of argument) that for 
	given $(I_0,S_0,\nu^0)\in \mathfrak U$, 
	system \eqref{2} with the Preisach operator \eqref{P} has a unique local solution
	with the initial data $I(0)=I_0, S(0)=S_0$ and the initial state $\nu^0$ (see, for example, the survey in Leonov et al. 2017). Further, the invariance of $\mathfrak D$ implies that each solution is extendable to the whole semi-axis $t\geq 0$. These solutions induce a {semi-flow} in the set $\mathfrak U$, which is considered to be the phase space 
	of system \eqref{2} and is endowed with a metric by the natural embedding into the space $\mathbb{R}^2\times L_1(\Pi;\mathbb{R})$. This leads to the standard definition of local and global stability including stability of equilibrium states and periodic solutions. In particular, an equilibrium is a triplet $(I_0,S_0,\nu^0)\in \mathfrak{U}$ and a periodic solution is a periodic function $(I(\cdot),S(\cdot),\nu(\cdot)): \mathbb{R}_+\to \mathfrak U$ where the last component viewed as a function $\nu:\mathbb{R}_+\times \Pi\to \{0,1\}$ of two variables $t\in\mathbb{R}_+$ and $\alpha\in\Pi$ is given by \eqref{re}.

The vaccination rate \eqref{pre} at an equilibrium is constant, while for a periodic solution the vaccination rate is also periodic with the period of $I$ and $S$.

\subsection{Hysteresis loops.}
We call a periodic input $I=I(t)$ {\em simple} if each local minimum of this input is its global minimum and each local maximum is its global maximum.
That is, the input increases from its global minimum value to its global maximum value and { then} decreases back to its global minimum value over one period.
Let us consider inputs $I:\mathbb R_+\to [0,1]$.

The following property of the Preisach operator is called {\em monocyclicity}:
for any $T$-periodic input $I(t)$, $t\ge 0$, and any admissible initial state of the Preisach operator, the corresponding output $v(t)=(\mathcal P[\nu^0]I)(t)$ satisfies $v(t+T)=v(t)$ for all $t\ge T$. Further, for a simple periodic input, the output is defined by
\begin{equation}\label{loop}
v(t)=\left\{
\begin{array}{lll}
\bar v (I(t)) & \text{as $I(t)$ increases},\\
\hat v (I(t)) & \text{as $I(t)$ decreases}
\end{array}
\right.
\end{equation}
for $t\ge T$, where the Lipschitz continuous functions $\bar v(\cdot), \hat v(\cdot)$ increase and satisfy
\begin{equation}\label{vv}
\bar v(I) < \hat v(I) \ \ \text{for} \ \ I_1<I<I_2;
\end{equation} 
\begin{equation}\label{vv''}
\bar v(I_1) = \hat v(I_1); \qquad \bar v(I_2) = \hat v(I_2)
\end{equation}
with
\begin{equation}\label{ii}
I_1:=\min_{t\ge 0} I(t), \qquad  I_2:=\max_{t\ge 0} I(t).
\end{equation}
In other words, a simple periodic input produces a closed (hysteresis) 
loop formed by the graphs of the functions $\bar v(\cdot), \hat v(\cdot)$ on the input-output diagram (after the moment $T$), see Figure \ref{fig3}(a). 

\begin{figure}[ht]
	\begin{center}
		\includegraphics[width=\linewidth]{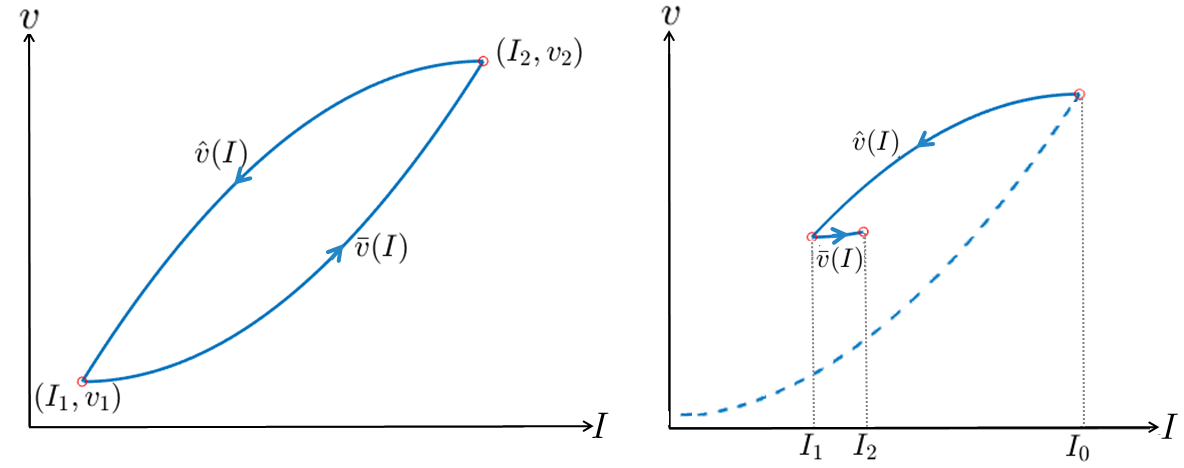}
	\small (a) \hfil (b)
	\end{center}
	\caption{Hysteresis loops on the $(I,v)$ diagram.}
	\label{fig3}
\end{figure}

Importantly, these functions depend on the initial state of the Preisach operator. However, 
by the definition of this operator,
\begin{equation}\label{width}
\hp(I) - \bar v(I) = \int_I^{I_2}\dd \alpha_2 \int_{I_1}^{I} q(\alpha_1, \alpha_2) \dd \alpha_1.
\end{equation}
Therefore, for a simple periodic input, the difference $\Delta v= \hat v-\bar v$ is a nonnegative function of three scalar variables, $\Delta v=\Delta v(I,I_1,I_2)$, defined on the domain $0\le I_1\le I\le I_2\le 1$.
Since $q=q(\alpha_1,\alpha_2)$ is bounded, the following quantity is well-defined and finite:
\begin{equation}\label{L}
L:=\max_{0\leq I_1  < I_2\leq 1}  \frac{1}{I_2 - I_1}  \left( \max_{I_1 \leq I\leq I_2}{\bigl(\hp(I) - \bar v(I)\bigr)}\right).
\end{equation}
In particular,
$L\leq K$ (cf.~\eqref{K}).
The quantity  \eqref{L} measures the maximal width of the hysteresis loops of the Preisach operator relative to their length. It plays an important role in the following. We will assume that 
\begin{equation}\label{xxx}
L < \beta.
\end{equation}

More generally, let $0\le t_0<t'<t''$ and let us consider an input $I=I(t)$, which increases on the interval $[t_0,t']$ and then decreases on the interval $[t',t'']$. Such inputs will be also called {\em simple} on the interval $[t_0,t'']$. Suppose that 
the input values $I_0=I(t_0)$, $I_2=I(t')$ and $I_1=I(t'')$ satisfy $I_0\leq I_1<I_2$. Then it follows from the definition of the Preisach operator that formula \eqref{loop} is valid, where the increasing functions $\bar v(\cdot), \hat v(\cdot)$ satisfy \eqref{vv}, but  relations \eqref{vv''} and \eqref{width} do not necessarily hold.
Formulas \eqref{loop} and \eqref{vv} are also true if the input $I=I(t)$ first decreases on the interval $[t_0,t']$ and then increases on the interval $[t',t'']$, and the input values
$I_0=I(t_0)$, $I_1=I(t')$ and $I_2=I(t'')$ satisfy $I_1<I_2\leq I_0$, see
see Figure \ref{fig3}(b).
Functions $\bar v=\bar v(I)$, $\hat v=\hat v(I)$ in \eqref{loop}
will be referred to as an {\em ascending branch} and a {\em descending branch}
of the Preisach operator. It is important to notice that these functions depend on the states $\nu_\alpha(t_0)$, $\alpha\in \Pi$, of the relays at the moment $t_0> 0$, which in turn depend on the initial states $\nu_\alpha^0$ of the relays at the moment $t=0$ and the value of the input $I$ on the interval $[0,t_0]$ (cf.~\eqref{re}). As such, on any interval of monotonicity of the input, the input-output pair $(I,v)$ follows one of infinitely many possible branches of the Preisach operator, and a particular branch followed by the input-output pair is uniquely defined by the prior history of the input variations and the initial states of the relays. 


\subsection{Equilibrium points.}
We will assume that
\begin{equation}\label{Pi}
\Pi = \{\alpha=(\alpha_1,\alpha_2): 0\le \alpha_1<\alpha_2\le 1 \}.
\end{equation}
Due to this assumption and the compatibility constraint \eqref{compatibility}, the inclusion $\alpha\in\Pi$ implies that all the relays are in state $\nu_\alpha=0$ when $I=0$. Therefore,
system \eqref{2}, \eqref{pre} has a 
unique {\em infection-free} equilibrium
\begin{equation}\label{infectionfree}
(I_*,S_*,\nu^0)=\left(0,\frac{\mu}{\mu+v_{nat}},0\right)\in \mathfrak{U},
\end{equation} 
in which the state $\nu^0=\nu_\alpha^0: \Pi\to\{0,1\}$ of the Preisach operator is the identical zero. According to \eqref{pre}, the vaccination rate at this equilibrium is minimal and equals $v_{nat}$.

In addition, if 
\begin{equation}\label{ineq}
R_0:=\frac{\beta\mu}{\delta(\mu+v_{nat})}>1,
\end{equation}
then system \eqref{2}, \eqref{pre} also has a family of {\em endemic} equilibrium states
\begin{equation}\label{endemic}
(I^*,S^*,\nu^0)=\left(\frac{\mu}{\delta}-\frac{\mu+v^0}{\beta},\frac{\delta}{\beta},\nu^0\right)\in \mathfrak{U}
\end{equation}
with $(I^*,S^*)\in \mathfrak D$, where
the vaccination rate $v^0$ is related to the state $\nu^0=\nu_\alpha^0: \Pi\to\{0,1\}$ of the Preisach operator by
\begin{equation}\label{vnu}
v^0=v_{nat}+\iint_{\Pi} \nu_{\alpha}^0 \,q(\alpha) \,d\alpha_1 d\alpha_2.
\end{equation}
These equilibria form a connected set in $\mathfrak U$. 
Further, the set \eqref{endemic}
includes equilibrium states with different vaccination rates $v^0$ and different proportions of the infected and recovered populations,

while the fraction of the susceptible individuals is the same at 
all these states.

\begin{remark}
	\rm If $\Pi$ is different from \eqref{Pi} and includes {points} $(\alpha_1,\alpha_2)$ with $\alpha_1<0, \ 0\le \alpha_2\le 1$, then the set of infection-free equilibrium states is also infinite and connected. Depending on the parameters, it can be either an attractor or repeller or include equilibrium states with different stability properties (this is not unlike the classical SIR model with zero mortality rate). Further, in this case, a solution which starts from small infected population and converges to an infection-free equilibrium is characterized by higher vaccination rate at the end than at the beginning. This is because the relays with $\alpha_1<0$ switch from state $0$ to state $1$ but never switch back due to the positivity of the input $I$.
\end{remark}

Trajectories of system \eqref{2}, \eqref{pre} lie in the infinite-dimensional phase space $\mathfrak U$ of triplets $(I,S,\nu^0)$.
Slightly abusing the notation, we will sometimes also refer to the two-dimensional curve  $(I(t),S(t))$ as to a trajectory, omitting the component \eqref{re} in the state space of the Preisach operator. 

	\begin{proposition}\label{p1}
	If 
	\begin{equation}\label{3}
	R_0=\frac{\beta\mu}{\delta(\mu+v_{nat})}\le 1,
	\end{equation}
	then the  infection-free equilibrium \eqref{infectionfree} is the global attractor of system \eqref{2}, \eqref{pre}. If the opposite inequality \eqref{ineq} holds,
	then any trajectory of system \eqref{2}, \eqref{pre}, which has at most finite number of intersections with the nullcline $\dot I=0$ (the line $S=S^*=\delta/\beta$), 
	converges to an endemic equilibrium.
\end{proposition}


	\begin{proof}
		
		If $\delta/\beta\ge 1$, then the first equation of \eqref{2} implies that $\dot I<0$ in $\mathfrak D$.
		Therefore, on any given trajectory of \eqref{2}, \eqref{pre}, the vaccination rate is defined by
		$v(t)={\hat v(I(t))}$ (cf.~\eqref{loop}), where the function ${\hat v(I)}$ is a continuous descending branch of the Preisach operator; this branch depends on 
		the initial state $\nu^0=\nu^0_\alpha$. Thus, a trajectory of \eqref{2}, \eqref{pre} is simultaneously a trajectory of the ordinary differential system
		 \begin{equation}\label{2'}
		\begin{array}{l}
		\dot I=\beta IS - \delta I,\\
		\dot S= -\beta IS - \tilde v(I)S-\mu S+\mu
		\end{array}
		\end{equation}
		with $\tilde v(\cdot)={\hat v(\cdot)}$ depending on  $\nu^0=\nu^0_\alpha$. Since $\dot I<0$, each trajectory of \eqref{2'} converges to the infection-free equilibrium $(I_*,S_*)$ for any branch $\tilde v(\cdot)={\hat v(\cdot)}$, and the result follows.
		
			If $1>\delta/\beta$ and \eqref{3} holds, then the second equation of \eqref{2} implies $\dot S<0$ for $S\ge \delta/\beta=S^*$ in $\mathfrak{D}$. Hence, all trajectories of \eqref{2}, \eqref{pre} enter the domain
$\{(I,S)\in \mathfrak D: S<S^*\}$ and remain there for all sufficiently large $t$.
In this domain, $\dot I<0$ as follows from the first of equations \eqref{2}.
Therefore, the same argument as we used in the case $\delta/\beta\ge 1$ above shows that all the trajectories of system \eqref{2}, \eqref{pre} converge to the infection-free equilibrium \eqref{infectionfree}.

Finally, assume that \eqref{ineq} holds, and suppose that a trajectory of  system \eqref{2}, \eqref{pre} has at most finite number of intersections with the line $S=S^*$ where $\dot I=0$. Then, after the last intersection, the $I$-component of the trajectory either strictly decreases or strictly increases with $t$.
In either case, the monotonicity of $I(t)$ implies that the trajectory of
the \eqref{2}, \eqref{pre} (after its last intersection with the line $S=S^*$) is simultaneously a trajectory of the ordinary differential system \eqref{2'} where $\tilde v(\cdot)$ is either a descending or an ascending branch of the Preisach operator.
Due to the fact that any branch $\tilde v(I)$ increases, relation \eqref{ineq} implies that
system \eqref{2'} has a unique endemic equilibrium $(I^*,S^*)$ given by
\begin{equation}\label{j1}
\mu (1-S^*)-\beta S^* I^* = S^* \tilde v(I^*), \qquad S^*=\delta/\beta.
\end{equation}
Furthermore, system \eqref{2'}
has a global Lyapunov function (Korobeinikov et al. 2002, Korobeinikov et al. 2004) 
\begin{equation}\label{6}
V(I,S)= S - S^* \ln \frac{S}{S^*}  + I - I^* \ln \frac{I}{I^*} + \frac1\beta \int_{I^*}^I \frac {\tilde v(i)-\tilde v(I^*)}{i}\,d i.
\end{equation}
Indeed, 
\begin{align*}
\dot V = \left (S- {S^*}\right) \left(-\beta I -\tilde v(I)-\mu +\frac{\mu}{S}\right)
+ \left (I- {I^*}\right) (\beta S-\delta) \\+(\tilde v(I)-\tilde v(I^*)) \left(S-\frac{\delta}{\beta}\right),
\end{align*}
where we replace $\mu =-\beta I^* -\tilde v(I^*) +{\mu}/{S^*}$, $\delta =\beta S^*$ to obtain
\[
\dot V = \left (S- {S^*}\right) \left(-\beta (I-I^*) -\tilde v(I) +\tilde v(I^*)-\frac{\mu}{S^*} +\frac{\mu}{S}\right)
+ \beta \left (I- {I^*}\right) (S-S^*)
\]
\begin{equation}\label{j2}
+(\tilde v(I)-\tilde v(I^*)) \left(S-S^*\right)=-\frac{\mu(S-S^*)^2}{S^* S}<0.
\end{equation}
This implies {convergence} to the endemic equilibrium point.
\end{proof} 

\section{Main Results}

\subsection{Poincar\'e-Bendixson type alternative}
Proposition \ref{p1} does not cover those trajectories that have infinitely many intersections with the nullcline $S=S^*$ in the case when relation \eqref{ineq} holds.

\begin{theorem}\label{t1} Let \eqref{ineq} hold.
Any trajectory of system \eqref{2}, \eqref{pre} (starting from any initial 
values from the region $I>0$, $S>0$, $S+I \le 1$ 
and any admissible initial state of the Preisach operator) converges either to an endemic equilibrium or to a simple periodic orbit.
\end{theorem}

\proof

Consider a trajectory $(I(t),S(t))$ which does not converge to an equilibrium point.
Due to Proposition \ref{p1}, it has infinitely many intersections with the line $S=S^*$ at points $I(t_k)=I_k$, $k=1,2,...$\ with $t_k<t_{k+1}$ (where we can assume without loss of generality that $I_1> I_2$).

If for some $k'$ we have $I_{k'} < I_{k'+2} < I_{k'+1}$, then 
let us compare the arc $\Gamma_{k'+2}$ of the trajectory $(I(t), S(t))$  
connecting the points $(I_{k'+2}, S^*)$ and $(I_{k'+3},S^*)$ with its arc $\Gamma_{k'}$ connecting the points $(I_{k'},S^*)$ and $(I_{k'+1},S^*)$. 
{Note that on each arc $\Gamma_{k'}$ the 
	vaccination rate follows a particular branch of the Preisach operator, which we will denote as $\bar v_{k'}$} (cf.~\eqref{loop}).

Both arcs $\Gamma_{k'+2}$ and $\Gamma_{k'}$ lie above the nullcline $S=S^*$, both go from left to right (hence $I_{k'+3}> I_{k'+2}$), and $\Gamma_{k'}$ starts to the left of $\Gamma_{k'+2}$. Since for the internal points of these arcs,
\[
\frac{dS}{dI}=\frac{\dot S}{\dot I}= \frac{-\beta IS -\bar v_i(I)S-\mu S+\mu}{\beta IS - \delta I},\qquad (I,S)\in \Gamma_i, \ i=k', k'+2,
\]
and $\bar v_{k'+2}(I)>\bar v_{k'}(I)$, 
we see that
$\Gamma_{k'}$ and $\Gamma_{k'+2}$ cannot intersect except at the end point, hence $I_{k'}<I_{k'+2}< I_{k'+3}\le I_{k'+1}$ as required.
Further, if $I_{k'+3}= I_{k'+1}$,
 then due to forward uniqueness the trajectory becomes periodic starting from the moment $t=t_{k'+1}$, i.e. $I_{k'+2j+1}=I_{k'+1}$, $I_{k'+2j+2}=I_{k'+2}$ for all $j=1,2,\ldots$
 
 Similarly, relations $I_{k'+1} > I_{k'+3} > I_{k'+2}$ imply
 $I_{k'+1} > I_{k'+3} > I_{k'+4}\ge I_{k'+2}$, and if 
 $I_{k'+4}= I_{k'+2}$, then the trajectory becomes periodic after the {moment}
 $t=t_{k'+2}$.

Combining the above two results, we see that
if either $I_{k'} < I_{k'+2} < I_{k'+1}$ or $I_{k'+1} > I_{k'+3} > I_{k'+2}$ for some $k'$, and the trajectory does not become periodic, 
then
\[
{I_{k'} < } I_{k'+2} < I_{k'+4} < I_{k'+6}< \cdots < I_{k'+5} < I_{k'+3} < I_{k'+1}.
\]
Therefore, the trajectory converges to a periodic orbit oscillating between the points $(I',S^*)$ and $(I'',S^*)$ with $I'=\lim_{j\to \infty} I_{k'+2j}$ and 
$I''=\lim_{j\to \infty} I_{k'+2j+1}$ (or an equilibrium if the two limits coincide).

The only remaining alternative to this scenario is to have either
	\[
	\cdots > I_5 > I_3 > I_{1} >  I_{2} > I_{4} > I_6 > \cdots
	\]
	or
	\[
	\cdots < I_5 < I_3 < I_{1} <  I_{2} < I_{4} < I_6 < \cdots
	\]
for all $I_k$.
In this case, again, the limit is a periodic trajectory oscillating between the points $(I',S^{*})$ and 
$(I'',S^*)$ unless $\inf I_k=0$. However, it is easy to see that actually $\inf I_k\ge \varepsilon>0$. Indeed, on each arc $\Gamma_k$ which lies below the line $S=S^*$, the component $I(t)$ of the solution decreases, and the vaccination rate is given by $v(t)=\hat v_k(I(t))$ for some descending branch $\hat v_k(\cdot)$ of the Preisach operator.
Therefore, the Lyapunov function $V_k(I,S)$ given by \eqref{j1}, \eqref{6} with $\tilde v(\cdot)=\hat v_k(\cdot)$ decreases  along the segment $\Gamma_k$ of the trajectory, hence the value of $V_k(\cdot,\cdot)$ at the left end of $\Gamma_k$ is less than at the right end.
But the functions $V_k(\cdot,S^*): (0,1]\to \mathbb{R}$ are uniformly bounded {for $I\in [\delta^{*},1]$}, {$\delta^{*} > 0$}, 
and satisfy
$
V_k(I,S^*)\ge -I^*_k \ln ({I}/{I^*_k})$, where $I^*_k$ is the solution of \eqref{j1} for $\tilde v(\cdot)=\bar v_k(\cdot)$.
Relation \eqref{ineq} ensures that $\inf I_k^*\ge \varepsilon_0>0$, hence $\inf_k V_k(I ,S^*)\to \infty$ as $I\to 0+$, and consequently the fact that the set of values of $V_k(\cdot,\cdot)$ at the right ends of the arcs $\Gamma_k$ is bounded implies that the left ends satisfy $\inf I_k\ge \varepsilon>0$.
\hfill $\Box$

\subsection{Sufficient conditions for global stability of the set of endemic equilibrium states.}

In the rest of the paper, we derive sufficient conditions which ensure the global convergence to endemic equilibrium states. 
We make the following assumption:

\medskip	
{\bf (A)}  
Each input-output loop of the Preisach operator corresponding to a simple periodic input is  {\em convex}. 
In other words, 
the function ${\bar v(\cdot)}$ in \eqref{loop} is convex and the function $\hat v(\cdot)$ in \eqref{loop} is concave. 

\begin{lemma}\label{l1}  
Let $\bar I^*>0$ and the function $\bar V(I,S)$ be defined by formulas
\eqref{j1}, \eqref{6} with $\tilde v(i)=\bar v(i)$, and
let $\hat I^*$, $\hat V(I,S)$ be defined by the same formulas 
with $\tilde v(i)=\hat v(i)$.
Assumption {\rm (A)} guarantees  that all the level sets of the function $\bar V$ are convex and the intersection of each level set of the function $\hat V$ with the half plane $I\le \hat I^*$ is convex. \textcolor{blue}{}
\end{lemma}

\proof
 The curvature of the level line of the function $V(I,S)$ is given by 
\[
\kappa = - \frac{V_{SS} (V_I)^2 + V_{II} (V_S)^2 - 2 V_{IS} V_I V_S}{(V_S^2 + V_I^2)^{3/2}}.
\]
We need to show that $\kappa < 0$, i.e.~$V_{SS} (V_I)^2 + V_{II} (V_S)^2 > 0$, which for the function \eqref{6} is equivalent to
\begin{equation}\label{convexity}
0 < \frac{S^*}{S^2} \, V_I^2 + \frac{  \beta I^* + I \tilde v'(I)-\tilde v(I) + \tilde v(I^*)}{ \beta I^2} \, V_S^2.
\end{equation}
For $\tilde v(\cdot)=\bar v(\cdot)$, the convexity of $\bar v$ implies
\[
({I-\bar I^*})\bar v'(I)\ge {\bar v(I)-\bar v(\bar I^*)},
\]
hence 
\[
I \bar v'(I) \ge \bar v(I)-\bar v({\bar I^*})
\]
(because $\bar v$ increases) and \eqref{convexity} follows. For $\tilde v(\cdot)=\hat v(\cdot)$, relation \eqref{convexity} holds in the region $I\le \hat I^*$ because $\hat v(\cdot)$ is an increasing function.
\hfill $\Box$


\begin{theorem}\label{t2}
Let assumptions {\rm (A)} and \eqref{ineq} be satisfied. 
Let the quantity \eqref{L} be sufficiently small. Then, system \eqref{2}, \eqref{pre} has no periodic solutions, and every trajectory converges to an endemic equilibrium point.
\end{theorem}

Due to Theorem \ref{t1}, it suffices to show that system \eqref{2}, \eqref{pre} has no simple periodic solutions if the quantity  $L$ defined by \eqref{L} is sufficiently small. 
An explicit estimate for $L$ will be established in the proof.

\subsection{Discussion}
To explore how the heterogeneity of the response of the susceptible population to the advent of an epidemic can affect dynamics of system \eqref{2}, we consider the aggregate vaccination rate \eqref{pre} with the Gaussian density
\begin{equation}\label{q}
q(\alpha_1,\alpha_2) = A e^{-\frac{(\alpha_1-\alpha_{m_1})^2 + (\alpha_2-\alpha_{m_2})^2}{2\sigma^2}},
\end{equation}
where the normalizing parameter
$A=A(\alpha_{m_1},\alpha_{m_2},\sigma)$ ensures that the integral of $q$ over the domain \eqref{Pi} equals $1$. The limit $\sigma\to 0$ corresponds to the perfectly homogeneous response \eqref{rev}, i.e. the vaccination rate switches from the value $v_{nat}=0$ to the value $v_{int}=1$ at the switching threshold $I=\alpha_{m_2}$ and switches backwards at the threshold $I=\alpha_{m_1}$. Increasing the variance $\sigma^2$ of the (truncated) Gaussian distribution corresponds to increasing the heterogeneity of the response within the susceptible population. 

Figure \ref{fig4} presents an example of the convergence to a periodic cyclic behavior for small $\sigma>0$.
This scenario for system \eqref{2}, \eqref{rev} with one non-ideal relay was studied in Chladn\'a et al. (2020). Further, Figure \ref{fig4} shows that as $\sigma$ increases, the periodic behavior is replaced with the convergence to an endemic equilibrium. This is in agreement with Theorem \ref{t2} because the quantity \eqref{L} decreases with increasing $\sigma$. Indeed, Theorem \ref{t2} ensures the global stability of the set of endemic equilibrium states when the quantity \eqref{L} becomes sufficiently small. 

The trajectory corresponding to the more heterogeneous response ($\sigma=0.01$) in Figure \ref{fig4} converges to an endemic equilibrium state with the densities $I^*=0.00063$, $S^*=0.056$ of the infected and susceptible populations, respectively. The vaccination rate at this equilibrium is $v^*=0.0041$ (week$^{-1}$). The  trajectory corresponding to the more homogeneous response ($\sigma=0.0009$), which converges to the periodic orbit, exhibits a lower initial peak of the infected population during the transient than the trajectory of the heterogeneous system. The average density of the susceptible population for the periodic trajectory, $\bar S=0.057$, is close to $S^*$. The density of the infected population along the periodic trajectory is much higher than $I^*$ at its peaks, but the average density  $\bar I=0.00016$ is significantly lower than $I^*$. This agrees with the fact that 
the average vaccination rate $\bar v=0.023$ (week$^{-1}$) is significantly higher than $v^*$. 

\begin{figure}
  \centering
  \begin{tabular}{@{}c@{}}
    \includegraphics[width=\linewidth]{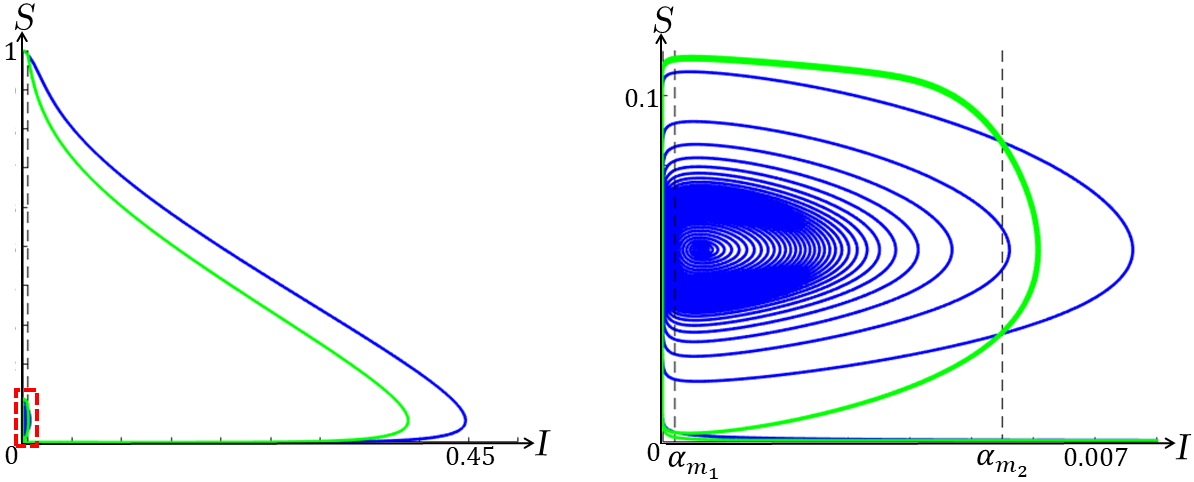} \\[\abovecaptionskip]
    \small (a) \hfil (b)
  \end{tabular}

  \vspace{\floatsep}

  \begin{tabular}{@{}c@{}}
    \includegraphics[width=\linewidth]{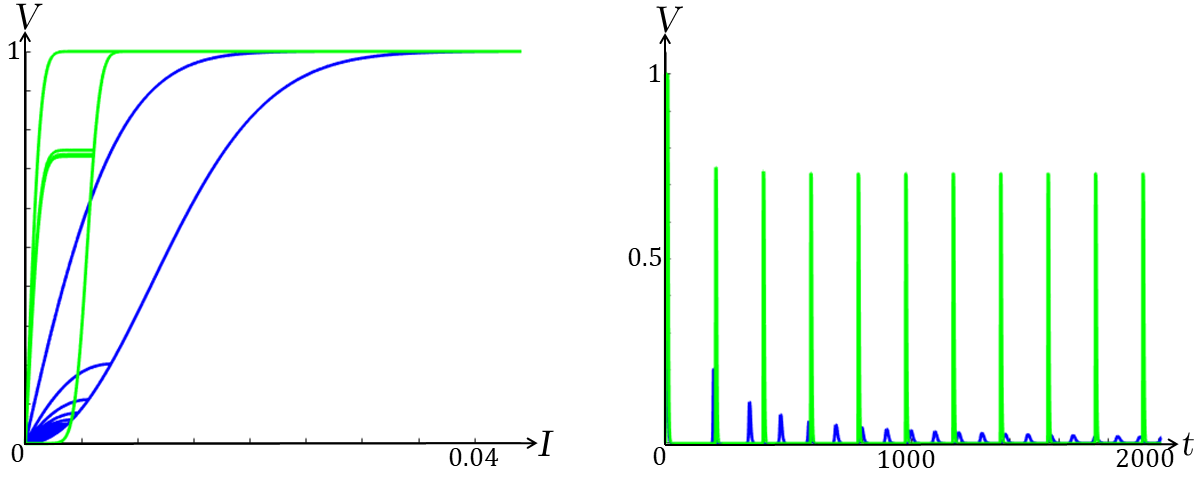} \\[\abovecaptionskip]
    \small (c) \hfil (d)
  \end{tabular}

  \vspace{\floatsep}

  \begin{tabular}{@{}c@{}}
    \includegraphics[width=\linewidth]{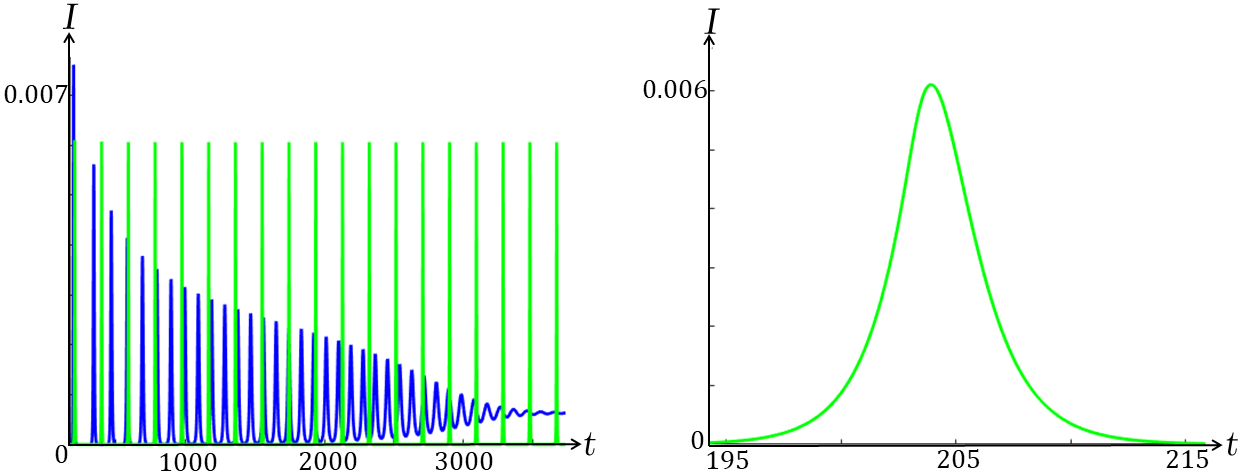} \\[\abovecaptionskip]
    \small (e) \hfil (f)
  \end{tabular}
  \caption{A solution of system \eqref{2} with the vaccination rate defined by relations \eqref{rev} and \eqref{q}. The parameters are $\mu = 0.0006$, $\delta = 0.6$, $\beta = 10.8$ (week$^{-1}$), $\alpha_{m_1} = 0.0002$, $\alpha_{m_2} = 0.0055$. The corresponding basic reproduction number is $R_0=18$. {\bf (a)} The green and blue trajectories correspond to $\sigma = 0.0009$ and $\sigma = 0.1$, respectively. The initial conditions for both trajectories, $I(0)=10^{-5}, S(0)=1-I(0)$, correspond to a small number of infected individuals in a fully susceptible population. {\bf (b)} Zoom into the region marked by the red box on panel (a). The green trajectory converges to a cycle; the blue trajectory converges to an endemic equilibrium. {\bf (c)} Hysteresis loops on the $(I, v)$-plane for both trajectories using the same color code.
{\bf (d, e)} Time traces of the infected population and the vaccination rate. The time unit is one week. {\bf (f)} Zoom into one pulse of the infected population for the green trajectory.}\label{fig4}
\end{figure}

\section{Proof of Theorem \ref{t2}}
The following inequalities will be systematically used: 
\begin{equation}\label{log1}
\ln (1 + x) \geq \frac{x}{1+x}, \qquad x \geq -1;
\end{equation}
\begin{equation}\label{log2}
\frac{x(2+x)}{2(1+x)}\geq   \ln (1 + x) , \qquad x \geq 0;
\end{equation}
\begin{equation}\label{log3}
\ln (1 + x) \leq \frac{2x}{2+x}, \qquad -1< x \leq 0.
\end{equation}

We prove the theorem by contradiction. Let us assume that there exists a simple periodic solution $(I(t), S(t))$, for which $I(t)$ increases from $I_1$ to $I_2$ and then decreases from $I_2$ to $I_1$ on a period with  $I_1 <I_2$.
Let us denote by $\bar v(I)$ and $\hat v(I)$ the two branches of the Preisach operator corresponding to the increasing $I$ and decreasing $I$ of this solution, respectively, hence relations 
\eqref{loop}\,--\,\eqref{ii}
hold.
Then system \eqref{2}, \eqref{pre} has an endemic equilibrium $(\bar I^*,S^*)$ defined by equation \eqref{j1} with $\tilde v(\cdot)=\bar v(\cdot)$ and 
an endemic equilibrium $(\hat I^*,S^*)$ defined by equation \eqref{j1} with $\tilde v(\cdot)=\hat v(\cdot)$.
Further, since $(I_1, S^*)$ and $(I_2, S^*)$ are the turning points of the periodic solution,
\begin{equation*}
\bar v(I_1)=\hat v(I_1) < \mu \left(\frac{\beta}{\delta} - 1\right) - \beta I_1; \quad \ \ \mu \left(\frac{\beta}{\delta} - 1\right) - \beta I_2  < \bar v(I_2)=\hat v(I_2).
\end{equation*}
On the other hand, according to \eqref{j1}, 
\[
\bar v(\bar I^*)=\mu \left(\frac{\beta}{\delta} - 1\right) - \beta \bar I^*, \quad \ \
\hat v(\hat I^*)=\mu \left(\frac{\beta}{\delta} - 1\right) - \beta \hat I^*.
\] 
Combining these four relations with \eqref{vv} and recalling that $\bar v$ and $\hat v$ are increasing functions, we see that
\begin{equation}\label{IhatI}
I_1 < \hat I^* < \bar I^* < I_2.
\end{equation}

Recall the definition of the functions $\bar V(I,S)$ and $\hat V(I,S)$, see Lemma \ref{l1}.
We now consider the following trivial identity
\begin{align}
\bigl( \bv(I_1,S^*)-  \bv(I_2,S^*) \bigr)  + \bigl(\hv(I_2,S^*) - \hv(I_1,S^*)  \bigr)  \nonumber \\= \bigl( \hv(I_2,S^*)-\bv(I_2,S^*)  \bigr)+ \bigl( \bv(I_1,S^*) - \hv(I_1,S^*) \bigr) \nonumber
 \end{align}
 and estimate the differences participating in it to establish a lower bound for $L$.


\subsection{Estimation of $\bv(I_1,S^*)-\bv(I_2,S^*)$. }
Dividing \eqref{j2} by the first equation of \eqref{2}, we obtain
\begin{equation}\label{j3}
\frac{d \bv}{d I} = -\mu \frac{S - S^*}{\beta I S S^*}= - \frac{\mu}{\beta I S^*}\left(1- \frac{S^*}{S}\right)
\end{equation}
along the trajectory.
For the part $\bar \Gamma$ of the periodic trajectory with increasing $I$ and $S>S^*$, this relation 
implies
\begin{equation}\label{monya}
\frac{d \bv}{d I}<0, \qquad I_1<I<I_2, \ (I,S)\in \bar \Gamma.
\end{equation}
Therefore,
this segment of the trajectory
lies outside the level set $\bv(I,S)\le \bv(I_2, S^*)$. The top point (with the largest $S$) of the level line 
$\bv(I,S)= \bv(I_2, S^*)$
is defined by $\partial \bv/ \partial I = 0$, i.e. $\beta I+ \bar v(I) = \beta \bar I^*+ \bar v(\bar I^*),$ which by monotonicity of $\bar v$ implies 
$I = \bar I^*$. Denoting the $S$-component of this point as $S_M$, we see that
\begin{equation}\label{j5}
\bv(I_2, S^*)= \bv(\bar I^*, S_M) = S_M - S^* \ln \frac{S_M}{S^*} + \bar I^*.  
\end{equation}


Let us fix a positive $h < S_M - S^*$.
Because the level set is convex (see Lemma \ref{l1}),
there is a unique point $(I_b,S_b)$ on the level line $\bar V(I,S)=\bv(I_2, S^*)$ with 
$S_b=S^*+h$, $I_b\ge \bar I^*$. 
The convexity of the level set
$\bar V(I,S)\le\bv(I_2, S^*)$ also implies that the point $(I_b,S^*+h)$ lies above the line segment
connecting the points $(\bar I^*,S_M)$ and $(I_2,S^*)$ because all three points lie on the boundary of this set. Therefore,  
\begin{equation}\label{nl2}
I_b - \bar I^* \ge \left( 1 - \frac{h}{S_M - S^*}\right) (I_2 - \bar I^*).
\end{equation} 
Since the part $\bar \Gamma$ of the periodic trajectory lies outside the level set
$\bar V(I,S)\le\bv(I_2, S^*)$ and the line segment $S=S^*+h$, $\bar I^*\le I\le I_b$ belongs to this set,
we have $S\ge S^*+h$ for  $(I,S)\in \bar\Gamma$, $\bar I^*\le I\le I_b$. Therefore, 
on the interval $[\bar I^*,I_b]$, relation \eqref{j3} implies
\[
\frac{d \bv}{d I} = - \frac{\mu}{\beta I S^*}\left(1- \frac{S^*}{S}\right) \leq - \frac{\mu}{\beta I S^*}\left(1- \frac{S^*}{S^* + h}\right) = - \frac{\mu h}{\beta S^* (S^* + h) I}.
\]
We integrate this inequality along the trajectory
over the interval $I \in [\bar I^*,I_b]$ and using the monotonicity of $\bv$ which follows from \eqref{monya}, we obtain
\[ 
\bv(I_2, S^*) - \bv(I_1, S^*) \leq \bv(I_b, S^*) - \bv(\bar I^*, S^*) \leq - \frac{\mu h}{\beta S^* (S^* + h)} \ln \frac{I_b}{\bar I^*},
\]
hence 
\begin{equation}\label{nl1}
\bv(I_1, S^*) - \bv(I_2, S^*) \geq \frac{\mu}{\beta S^*} \max_{h\le S_M - S^*}\frac{h}{S^* + h} \ln \frac{I_b}{\bar I^*}.
\end{equation}
Using inequality   \eqref{log1} and relation \eqref{nl2}, we estimate the right hand side of \eqref{nl1} as follows: 
\begin{align}\nonumber
&\frac{\mu}{\beta S^*} \max_{h\le S_M - S^*}\frac{h}{S^* + h} \ln \frac{I_b}{\bar I^*}  \geq \frac{\mu}{\beta S^*} \max_{h\le S_M - S^*}\frac{h}{S^* + h}\cdot \frac{I_b - \bar I^*}{I_b}\\ \nonumber
& \geq \frac{\mu}{\beta S^*} \max_{h\le S_M - S^*}\frac{h}{S_M}\left(1-\frac{h}{S_M - S^*}\right) \frac{I_2 - \bar I^*}{I_b} \\ \nonumber
& = \frac{\mu}{\beta S^*S_M {I_b} } (I_2 - \bar I^*) (S_M - S^*) \max_{h\le S_M - S^*}\frac{h}{S_M-S^*}\left(1-\frac{h}{S_M - S^*}\right)\\ \nonumber
& = \frac{\mu}{4\beta S^*S_M {I_b}} (I_2 - \bar I^*) (S_M - S^*).
\end{align}
Therefore, 
\begin{equation}\label{1stestimate}
\bv(I_1, S^*) - \bv(I_2, S^*) \geq \frac{\mu}{4\beta S^*S_M {I_b}} (I_2 - \bar I^*) (S_M - S^*).
\end{equation}

In order to estimate $S_M - S^*$, we evaluate the function $\bv$ at the point $(I_2,S^*)$
using \eqref{6} and substitute the result into \eqref{j5} to obtain
\begin{align}
S_M - S^* - S^* \ln \frac{S_M}{S^*}  = I_2 - \bar I^* - \bar I^* \ln \frac{I_2}{\bar I^*} + 
\frac{1}{\beta} \int_{\bar I^*}^{I_2} \frac{\bar v(i) - \bar v(\bar I^*)}{i} \dd i.\label{nl3}
\end{align} 
We  estimate the left hand side of \eqref{nl3}  from above  using
\begin{align*}
&S_M - S^* - S^* \ln \frac{S_M}{S^*} = S^* \left[\frac{S_M}{S^*} -1 -\ln\left(1 + \frac{S_M}{S^*} -1\right)\right] \\
&\leq 
S^*\frac{\left(\frac{S_M}{S^*} -1\right)^2}{\frac{S_M}{S^*} } = \frac{(S_M - S^*)^2}{S_M}.
\end{align*}
Further, we find a lower bound of the right hand side of \eqref{nl3} using estimate \eqref{log2}:  
\begin{align*}
I_2 - \bar I^* - \bar I^* \ln \frac{I_2}{\bar I^*} + \frac{1}{\beta} \int_{\bar I^*}^{I_2} \frac{\bar v(i) - \bar v(\bar I^*)}{i} \dd i \geq I_2 - \bar I^* - \bar I^* \ln \frac{I_2}{\bar I^*}\\
 = {\bar I^*}\left[\frac{I_2}{\bar I^*} - 1 - \ln\left(1 + \frac{I_2}{\bar I^*} - 1\right)\right] \geq \bar I^* \frac{\left(\frac{I_2}{\bar I^*} - 1\right)^2}{2\frac{I_2}{\bar I^*}} = \frac{(I_2 - \bar I^*)^2}{2I_2}.
\end{align*}
Combining the last two inequalities with \eqref{nl3}, we obtain 
\[
\frac{(S_M - S^*)^2}{S_M} \geq \frac{(I_2 - \bar I^*)^2}{2I_2}, 
\]
hence
\[
S_M - S^* \geq \sqrt{\frac{S_M}{2I_2}}(I_2 - \bar I^*).
\]
Substituting this relation into \eqref{1stestimate}, we finally arrive at
\begin{equation}\label{clubs1}
\bv(I_1, S^*) - \bv(I_2, S^*) \geq 
\frac{\mu}{4\beta S^* I_b \sqrt{2 S_M I_2}} (I_2 - \bar I^*)^2 \geq
\frac{\mu}{4\sqrt{2}\,\delta } (I_2 - \bar I^*)^2,
\end{equation}
where the last inequality holds due to $I_b\le I_2\le 1$, $S_M\le 1$ and $S^*=\delta/\beta$.

\subsection{Estimation of $\hv(I_1,S^*)-\hv(I_2,S^*)$. }
Now we consider the lower part $\hat \Gamma$ of the trajectory with decreasing
$I$ and $S<S^*$. We slightly modify the above argument. The relation
\begin{equation}\label{j6}
\frac{d \hv}{d I} = -\mu \frac{S - S^*}{\beta I S S^*}= - \frac{\mu}{\beta I S^*}\left(1- \frac{S^*}{S}\right),
\end{equation}
which is similar to \eqref{j3}, implies
\begin{equation}\label{monya'}
\frac{d \hv}{d I}>0, \qquad I_1<I<I_2, \ (I,S)\in \hat \Gamma.
\end{equation}
Hence, $\hat\Gamma$
lies outside the level set $\hv(I,S)\le \hv(I_1, S^*)$. 
The bottom point (with the smallest $S$) of the level line 
$\hv(I,S)= \hv(I_1, S^*)$
is defined by $\partial \hv/ \partial I = 0$, therefore $I = \hat I^*$ at this point. Denoting its $S$-component by $S_m$, we obtain
\begin{equation}\label{j7}
\hv(I_1, S^*)= \hv(\hi^*, S_m) = S_m - S^* \ln \frac{S_m}{S^*} + \hi^*,  
\end{equation}
which is similar to \eqref{j5}. We fix a positive $h < S^* - S_m$. Using the convexity of the intersection of the level set $\hv(I,S)\le \hv(I_1, S^*)$ with the half-space $I\le \hat I^*$ (see Lemma \ref{l1}), we establish the existence of a unique point on the level line $\hat V(I,S)=\bv(I_1, S^*)$ with the coordinates $(I_c,S_c)$ satisfying
$S_c=S^*-h$, $I\leq \hat I^*$.
Arguing as before (cf.~\eqref{nl2}), we obtain 
\begin{equation}\label{nl5}
\hi^* - I_c \ge  \left( 1 - \frac{h}{S^* - S_m}\right) (\hi^* - I_1).
\end{equation} 
Since $\hat \Gamma$ lies outside the level set
$\hat V(I,S)\le\hv(I_1, S^*)$,
the $S$-coordinate of the points $(I,S)\in \hat \Gamma$ with $I\in [I_c, \hat I^*]$ satisfies 
$S\le S^*-h$, hence \eqref{j6} implies
\[
\frac{d \hv}{dI} \geq  \frac{\mu h}{\beta S^* I (S^* - h)} ,
\]
which after integration over $ [I_c, \hi^*]$,  using also \eqref{monya'}, gives
\[\hv(I_2, S^*) - \hv(I_1, S^*) \geq \hv(\hi^*, S^*) - \hv(I_c, S^*) \geq \frac{\mu h}{\beta S^* (S^* - h)} \ln \frac{\hi^*}{I_c}.
\]
Therefore,
\begin{equation}\label{nl4}
\hv(I_2, S^*) - \hv(I_1, S^*) \geq \frac{\mu}{\beta S^*} \max_{h\le S^*-S_m}\frac{h}{S^* - h} \ln \frac{\hi^*}{I_c}.
\end{equation}
Using relation \eqref{log1} 
and \eqref{nl5}, the right hand side of  equation \eqref{nl4} can be estimated  as follows:
\begin{align*} 
&\frac{\mu}{\beta S^*} \max_{h\le S^*-S_m}\frac{h}{S^* - h} \ln \frac{\hi^*}{I_c} \geq \frac{\mu}{\beta S^*} \max_{h\le S^*-S_m}\frac{h}{S^*} \ln \left( 1 + \frac{\hi^* - I_c}{I_c}\right)\\
& \geq \frac{\mu}{\beta S^*} \max_{h\le S^*-S_m}\frac{h(\hi^* - I_c)}{S^* \hi^*} 
 = \frac{\mu}{\beta S^*} \max_{h\le S^*-S_m}\frac{h}{S^*\hi^*}\left(1-\frac{h}{S^*-S_m}\right)(\hi^* - I_1)\\
 &= \frac{\mu}{\beta (S^*)^2\hi^* } (\hi^* - I_1) (S^* - S_m) \max_{h\le S^*-S_m}\frac{h}{S^* - S_m}\left(1-\frac{h}{S^* - S_m}\right)\\
& = \frac{\mu}{4\beta (S^*)^2 \hi^*} (\hi^* - I_1) (S^* - S_m).
\end{align*}
Hence,
\begin{equation}\label{nl4'}
\hv(I_2, S^*) - \hv(I_1, S^*) \geq \frac{\mu}{4\beta (S^*)^2 \hi^*} (\hi^* - I_1) (S^* - S_m).
\end{equation}
In order to estimate $S^* - S_m$, we evaluate $\hat V(I_1,S^*)$ using \eqref{6} and rewrite \eqref{j7} equivalently as
\begin{align}
S^* - S_m + S^* \ln \frac{S_m}{S^*} = \hi^* - I_1 + \hi^* \ln \frac{I_1}{\hi^*} - 
\frac{1}{\beta} \int_{\hi^*}^{I_1} \frac{\hp(i) - \hp(\hi^*)}{i} \dd i.\label{nl6}
\end{align} 
A lower bound of the left hand side of \eqref{nl6} using \eqref{log1} 
is 
\[S^* - S_m + S^* \ln \frac{S_m}{S^*} \geq S^* - S_m + S^* - \frac{(S^*)^2}{S_m} = -\frac{(S^* - S_m)^2}{S_m}.
\]
An upper estimate of the right hand side of \eqref{nl6} using  \eqref{log3}  and the monotonicity of $\hat v$ 
is
\begin{align*}
&\hi^* - I_1 + \hi^* \ln \frac{I_1}{\hi^*} - \frac{1}{\beta} \int_{\hi^*}^{I_1} \frac{\hp(i) - \hp(\hi^*)}{i} \dd i \leq 
\hi^* - I_1 + \hi^* \ln \frac{I_1}{\hi^*} \\
&\leq \hi^* - I_1 + \hi^* \frac{2\left(\frac{I_1}{\hi^*}-1\right)}{1+\frac{I_1}{\hi^*}}=-\frac{(\hi^* - I_1)^2}{I_1 + \hi^*}.
\end{align*}
Combining the previous two inequalities with \eqref{nl6}, we obtain
\[
\frac{(S^* - S_m)^2}{S_m} \geq \frac{(\hi^* - I_1)^2}{I_1 + \hi^*}, 
\]
hence
\[
S^* - S_m \geq (\hi^* - I_1)\sqrt{\frac{S_m}{I_1 + \hi^*}},
\]
and \eqref{nl4'} implies
\begin{equation}\label{clubs2} 
\hv(I_2, S^*) - \hv(I_1, S^*)
\geq \frac{\mu}{4\beta (S^*)^2 \hi^*}\sqrt{\frac{S_m}{I_1 + \hi^*}}\, (\hi^* - I_1)^2
 \geq \frac{\mu \sqrt{S_m}}{4  \sqrt{{2}}\, \delta  } (\hi^* - I_1)^2,
\end{equation}
where we also use that $I_1\le \hat I^*\le 1$, $S^*\le 1$ and $\beta S^*=\delta$.

\subsection{Estimation of $ \bv(I_1,S^*) - \bv(I_2,S^*) + \hv(I_2,S^*)- \hv(I_1,S^*)$. }
Using the definition \eqref{6} of the Lyapunov function, we can write
\begin{align}\nonumber
&\bv(I_1,S^*) - \bv(I_2,S^*) + \hv(I_2,S^*)- \hv(I_1,S^*) = \nonumber \\
&= - \bar I^* \ln \frac{I_1}{\bar I^*}+\frac{1}{\beta} \int_{\bar I^*}^{I_1} \frac{\bar v(i) - \bar v(\bar I^*)}{i} \dd i 
+ \bar I^* \ln \frac{I_2}{\bar I^*}-\frac{1}{\beta} \int_{\bar I^*}^{I_2} \frac{\bar v(i) - \bar v(\bar I^*)}{i} \dd i   \nonumber \\
&- \hi^*\ln \frac{I_2}{\hi^*}+\frac{1}{\beta} \int_{\hi^*}^{I_2} \frac{\hp(i) - \hp(\hi^*)}{i} \dd i  
+ \hi^* \ln \frac{I_1}{\hi^*}-\frac{1}{\beta} \int_{\hi^*}^{I_1} \frac{\hp(i) - \hp(\hi^*)}{i} \dd i   \nonumber \\
&= \bar I^* \ln \frac{I_2}{I_1} - \hi^*\ln \frac{I_2}{I_1} -\frac{1}{\beta} \int_{I_1}^{I_2} \frac{\bar v(i)- \bar v(\bar I^*)}{i} \dd i  + \frac{1}{\beta} \int_{I_1}^{I_2} \frac{\hp(i) - \hp(\hi^*)}{i} \dd i  \nonumber \\
&= \left[(\bar I^* - \hi^*) + \frac{1}{\beta} \bar v(\bar I^*) -\frac{1}{\beta} \hp(\hi^*)\right]\ln \frac{I_2}{I_1} + \frac{1}{\beta} \int_{I_1}^{I_2} \frac{\hp(i) - \bar v(i)}{i} \dd i =  \nonumber \\ &= \frac{1}{\beta} \int_{I_1}^{I_2} \frac{\hp(i) - \bar v(i)}{i} \dd i, \label{nl7}
\end{align}
where we use the fact that 
\begin{equation}\label{nl8}
\beta(\bar I^* - \hi^*) + \bar v(\bar I^*) - \hp(\hi^*) = 0.
\end{equation}
Indeed, the  fixed points $(\hat I^*,S^*)$, $(\bar I^*,S^*)$  of \eqref{2}, \eqref{pre} (with $v=\hat v, \bar v$, respectively) satisfy the equations
\begin{align}
-\beta S^* \hi^* - \hp(\hi^*)S^* + \mu(1-S^*) = 0,\label{al1}\\
-\beta S^* \bar I^* - \bar v(\bar I^*)S^* + \mu(1-S^*) = 0,\nonumber 
\end{align}
and therefore taking their difference gives \eqref{nl8}.
Equations \eqref{nl7} imply  
\[
\bv(I_1,S^*) - \bv(I_2,S^*) + \hv(I_2,S^*)- \hv(I_1,S^*)  \leq \frac{I_2-I_1}{\beta I_1} \max_{I_1 \le I \le I_2}(\hp(I) - { \bar v(I)}).
\]
Now, we combine this relation with \eqref{clubs1} and \eqref{clubs2} to obtain 
\[
0 \leq -\frac{\mu}{ \sqrt{2}\,\delta } (I_2 - \bar I^*)^2  - \frac{\mu \sqrt{S_m} }{4\sqrt{2}\,\delta } (\hi^* - I_1)^2 
+ \frac{I_2-I_1}{\beta I_1} \max_{I_1 \le I \le I_2}(\hp(I) - \bar v(I))
\]
and further,
\begin{equation}
A (I_2 - \bar I^*)^2  +B (\hi^* - I_1)^2 \leq
\frac{L(I_2-I_1)^2}{\beta I_1}, \label{nl11}
\end{equation}
where $L$ is defined by \eqref{L} and 
\[
A := \frac{\mu}{4\sqrt{2}\,\delta}, \qquad
B := \frac{\mu\sqrt{S_m}}{4   \sqrt{2} \,\delta}.
\]
Since the function $\hat v$ increases, it follows from \eqref{IhatI} and \eqref{nl8} that
\[
\bar I^* - \hi^* \leq \frac{1}{\beta} (\hp({ \bar I^*}) - \bar v(\bar I^*)) \leq \frac{1}{\beta} \max_{I_1 \leq I  \leq I_2} (\hp(I) - \bar v(I))
\leq \frac{L}{\beta}(I_2 - I_1) ,
\]
i.e.
\begin{equation}
I_2-\bar I^*\ge 0, \quad \ \ \hat I^*-I_1\ge 0, \quad \ \ I_2-\bar I^* + \hat I^*-I_1\geq \Big(1-\frac{L}{\beta}\Big)(I_2-I_1),
 \label{nl10}
\end{equation}
where $L/\beta<1$ according to \eqref{xxx}.

Let us find the minimum value $F_{min}$ of the function
$F(x, y)=Ax^2+By^2$ under the constraints $x, y\ge 0$ and $x+y \ge (1- L/\beta)(I_2-I_1).$
Clearly, the minimum value is achieved for $x+y = (1- L/\beta)(I_2-I_1)$ and equals
\[
F_{min}=\frac{AB}{A+B}\left(1-\frac{L}{\beta}\right)^2 (I_2-I_1)^2.
\]
Hence, due to \eqref{nl10}, the left hand side of \eqref{nl11} satisfies
\[
A(I_2 - \bar I^*)^2 + B(I_1 - \hi^*)^2 \geq 
F_{min}
 \]
and \eqref{nl11} implies
\[
\frac{AB}{A+B}\left(1-\frac{L}{\beta}\right)^2 \leq \frac{L}{\beta I_1}.
\]
Recalling the definition of $A$ and $B$, this is equivalent to
\begin{equation}\label{j9}
I_1\,\frac{\beta}{L}  \left(1-\frac{L}{\beta}\right)^2 \leq \frac{4\sqrt{2}\,\delta}{\mu} \left(1+ \frac{1}{\sqrt{S_m}} \right).
\end{equation}

\subsection{
	Estimates of $I_1$, $S_m$.}
Denote
\begin{equation}\label{v}
v_{max}:= \iint_{0\le \alpha_1 <  \alpha_2\leq 1} q(\alpha_1,\alpha_2) \dd \alpha_1 \dd \alpha_2.
\end{equation}
It follows from \eqref{monya'}  that $\hv(I_1, S^*) \leq \hv(I_2, S^*)$. Using formula \eqref{6} for $\hat V$, this implies 
\begin{equation}\label{jadd}
I_1 - \hi^* \ln\frac{I_1}{\hi^*} \leq I_2 - \hi^* \ln\frac{I_2}{\hi^*} + \frac{v_{max}}{\beta}\ln\frac{I_2}{\hi^*} - \frac{\hp(\hi^*)}{\beta}\ln\frac{I_2}{\hi^*},
\end{equation}
where 
we use \eqref{v} and the estimate
\[
\int_{\hat I^*}^{I_1}\frac{\hat v(i)-\hat v(\hat I^*)}{i} \dd i\ge 0,
\]
which follows from the monotonicity of $\hat v$.
Relations  \eqref{al1}, \eqref{jadd} and $0<I_1<I_2 \leq 1$ imply that
\[
- \hi^* \ln\frac{I_1}{\hi^*} \leq 1 + \left(\frac{v_{max}}{\beta} - \frac{\mu(1-S^*)}{\beta S^*}\right)\ln\frac{I_2}{\hi^*},
\]
hence 
\begin{equation}\label{*1}
I_1 \geq \hi^* \exp{\left(-\frac{1}{\hi^*} - \frac{v_{max} S^*-\mu (1-S^*)}{\beta S^* {\hi^*}}\ln\frac{I_2}{\hi^*}\right)}.
\end{equation}
From relations \eqref{al1} and
\begin{equation}\label{LK}
\bar v(I)\le \hat v(I) \leq v_{nat} + K I
\end{equation}
(where \eqref{LK} follows from \eqref{LipP*})
it follows that 
\[
\frac{\mu(1-S^*)}{S^*} =\beta \hat I^*+\hat v(\hat I^*)\leq \beta \hat I^* +K \hat I^* +v_{nat},
\]
hence
\begin{equation}\label{*2}
\rho_0:=\frac{\mu(1-S^*)-v_{nat} S^* }{(\beta+K) S^*}\leq \hat I^*,
\end{equation}
where, due to assumption \eqref{ineq},
\[
 \rho_0>0.
\] 
%
%
If $v_{max} S^*\leq \mu (1-S^*)$, then relations \eqref{*1}, \eqref{*2} imply
\[
I_1 \geq \hi^*e^{-\frac{1}{\hi^*}} \geq \rho_0 e^{-\frac{1}{\rho_0}}.
\]
On the other hand, if $v_{max} S^* > \mu(1-S^*)$, then 
\begin{align*}
I_1  \geq  \hi^* \exp{\left(-\frac{\beta S^*+\bigr(v_{max}S^*-\mu(1-S^*)\bigl)\ln\frac{I_2}{\hi^*}}{ \beta S^*\hi^*}\right)} \\ \geq \rho_0 \exp{\left(-\frac{\beta S^*+\bigl(v_{max}S^*-\mu(1-S^*)\bigr)\ln\frac{1}{\rho_0}}{\beta S^* \rho_0}\right)}.
\end{align*}
Combining the two cases,
\begin{equation}\label{finalestimate}
I_1 \geq  \rho_1:= \rho_0    
\exp{\left(-\frac1\rho_0+ \frac{\ln{\rho_0}}{\beta S^* \rho_0}\big\lfloor v_{max}S^*-\mu(1-S^*)\big \rfloor_+ \right)
},
\end{equation}
where $\lfloor a\rfloor_+ = a$ for $a> 0$ and $\lfloor a\rfloor_+ = 0$ for $a\leq  0$.

\medskip
Finally, we obtain a lower bound for $S_m.$
From \eqref{nl6} and \eqref{al1} it follows that
\begin{align*}
-S^{*}  \ln \frac{S_m}{S^*} 
\le S^* -S_m +I_1- \hi^* -\hi^* \ln \frac{I_1}{\hi^*} - \frac{\hp(\hi^*)}{\beta}\ln \frac{I_1}{\hi^*} 
\\
\leq S^* - \frac{\mu  (1- S^*)}{\beta S^*} \ln \frac{I_1}{\hi^*}.
\end{align*}
Therefore,
\[ 
S_m \geq S^{*} \exp\left\{ -1 + \frac{\mu  (1- S^*)}{{\beta (S^*)}^{2}} \ln \frac{I_1}{\hi^*}\right\}
\]
and using \eqref{finalestimate} we arrive at
\begin{equation}\label{last}
S_m \geq S^{*} \exp\left\{ -1 + \frac{\mu  (1- S^*)}{{\beta (S^*)}^{2}} \ln \rho_1\right\}.
\end{equation}


Thus, we have shown that the existence of a simple periodic orbit implies estimates \eqref{j9}, \eqref{finalestimate} and \eqref{last}, which establish a lower bound $L\ge L_0>0$ on the quantity \eqref{L}. This completes the proof.

\section{Conclusions}

We considered an SIR model with vaccination, where we assumed that the vaccination rate changes in response to
dynamics of the epidemic. We modeled the adaptive response of an individual to the varying number of active cases
by a two-state two-threshold switch, and 
the aggregate response of the susceptible population by the Preisach operator.
This operator relationship between the vaccination rate and the number of active cases accounts for the heterogeneity of the response among the susceptible individuals.

If the basic reproduction number satisfies $R_0>1$, then the infection-free equilibrium is  globally stable.
On the other hand, if $R_0>1$, then the system has a connected infinite set of endemic equilibrium states. In this case, we showed that each trajectory converges either to an endemic equilibrium or to a periodic orbit. This is in agreement with Chladn\'a et al. (2020) where a simpler system with the homogeneous response modeled by a single two-state two-threshold switch was considered. Further, we showed that the set of endemic equilibrium states is the global attractor if a certain parameter of the Preisach operator, which is associated with the width of the hysteresis loops relative to their length, is sufficiently small. This parameter decreases with increasing heterogeneity of the adaptive response among the susceptible population. Based on these results, one can conclude that the heterogeneity of the response promotes the convergence to an endemic equilibrium state, while the homogeneous response may result in recurrent periodic outbreaks of the epidemic.






\end{document}